\documentclass[usenames,dvipsnames, 11pt]{article}
\usepackage[utf8]{inputenc}
\usepackage{mathtools}

\usepackage{amsmath, amssymb, amsthm, bm}

\usepackage[margin=2cm,a4paper]{geometry} 
\usepackage{tikz}
\usepackage{hyperref}
\usepackage{ulem}
\usepackage{enumerate}

\newcommand{\qbin}[2]{\genfrac{[}{]}{0pt}{}{#1}{#2}}

\DeclareMathOperator{\supp}{supp}
\newcommand{\F}{\mathbb{F}}
\newcommand{\Fq}{\mathbb{F}_q}
\newcommand{\transp}{\mathsf{T}}
\newcommand{\ent}{\operatorname{H}}
\newcommand{\C}{\mathcal{C}}
\newcommand{\wt}{\mathrm{wt}}

\newtheorem{lemma}{Lemma}[section]

\newtheorem{cor}[lemma]{Corollary}
\newtheorem{thm}[lemma]{Theorem}

\theoremstyle{definition}
\newtheorem{mydef}[lemma]{Definition}

\newtheorem{ques}[lemma]{Question}

\newtheorem{remark}[lemma]{Remark}

\begin{document}
\title{Blocking sets, minimal codes and trifferent codes}

\author{Anurag Bishnoi\thanks{Delft Institute of Applied Mathematics, Technische Universiteit Delft, 2628 CD Delft, Netherlands. E-mail: \textsf{A.Bishnoi@tudelft.nl}.} \and Jozefien D'haeseleer\thanks{Department of Mathematics: Analysis, Logic and Discrete Mathematics, Ghent University, Krijgslaan 281, 9000 Gent, Belgium. E-mail: \textsf{jozefien.dhaeseleer@ugent.be}. } \and Dion Gijswijt\thanks{Delft Institute of Applied Mathematics, Technische Universiteit Delft, 2628 CD Delft, Netherlands. E-mail: \textsf{D.C.Gijswijt@tudelft.nl}.} \and Aditya Potukuchi\thanks{Department of Electrical Engineering and Computer Science, York University
4700 Keele Street, Toronto, Ontario, Canada, M3J 1P3. E-mail: \textsf{apotu@yorku.ca}}}

\maketitle

\begin{abstract}
    We prove new upper bounds on the smallest size of affine blocking sets, that is, sets of points in a finite affine space that intersect every affine subspace of a fixed codimension.
    We show an equivalence between affine blocking sets with respect to codimension-$2$ subspaces that are generated by taking a union of lines through the origin, and strong blocking sets in the corresponding projective space, which in turn are equivalent to minimal codes. 
    Using this equivalence, we improve the current best upper bounds on the smallest size of a strong blocking set in finite projective spaces over fields of size at least $3$. 
    Furthermore, using coding theoretic techniques, we improve the current best lower bounds on strong blocking set.
    
    Our main motivation for these new bounds is their application to trifferent codes, which are sets of ternary codes of length $n$ with the property that for any three distinct codewords there is a coordinate where they all have distinct values. 
    Over the finite field $\mathbb{F}_3$, we prove that minimal codes are equivalent to linear trifferent codes. 
    Using this equivalence, we show that any linear trifferent code of length $n$ has size at most $3^{n/4.55}$, improving the recent upper bound of Pohoata and Zakharov. 
    Moreover, we show the existence of linear trifferent codes of length $n$ and size at least $\frac{1}{3}\left( 9/5 \right)^{n/4}$, thus (asymptotically) matching the best lower bound on trifferent codes. 
    
    We also give explicit constructions of affine blocking sets with respect to codimension-$2$ subspaces that are a constant factor bigger than the best known lower bound. 
    By restricting to~$\mathbb{F}_3$, we obtain linear trifferent codes of size at least $3^{23n/312}$, improving the current best explicit construction that has size $3^{n/112}$.\\

\end{abstract}

\section{Introduction}

A classic problem in finite geometry is to study sets of points that block every subspace of a specific dimension. 
This problem was first introduced in 1956 by Richardson~\cite{richardson1956finite}, who called such sets in finite projective spaces blocking coalitions.
We follow the now standard terminology of blocking sets~\cite{blokhuis2011blocking} and study the extremal problem of determining their minimum possible size. 
In combinatorial terminology, finding the smallest size of a blocking set is equivalent to determining the vertex cover number of the hypergraph that has points as its vertices and subspaces of the given dimension as its edges (see~\cite{furedi1988matchings} for a survey on covers of hypergraphs). 
In this paper, our main focus is on blocking sets in finite affine spaces and strong blocking sets in finite projective spaces.

For, $0 \leq s \leq k$ an $s$-blocking set in $\F_q^k$ is a set of points that contains at least one point from every \textit{affine} subspace of dimension $k - s$.
Let $b_q(k, s)$ denote the smallest possible size of an $s$-blocking set in~$\F_q^k$. 
Jamison~\cite{jamison1977covering}, and independently Brouwer and Schrijver~\cite{Brouwer1978}, proved that $b_q(k, 1) \geq (q - 1)k + 1$, using algebraic methods.
This is a foundational result for polynomial methods in combinatorics, and it is often shown as a corollary of the well-known combinatorial nullstellensatz~\cite{alon1999combinatorial, ball2011polynomial}, or the Alon-F\"uredi theorem~\cite{alon1993covering, bishnoi2018zeros}. 
Note that the lower bound of Jamison/Brouwer-Schrijver is easily seen to be tight by taking all points on the $k$ axes of $\F_q^k$.
Using a geometric argument (see for example \cite[Section 3]{ball2011polynomial}), the lower bound on $b_q(k,1)$ implies the following,
\begin{equation}\label{eq:blocking_lb_basic}
b_q(k, s) \geq (q^s - 1)(k - s + 1) + 1.    
\end{equation}

Unlike the $s = 1$ case, there is no $s > 1$ for which the bound in \eqref{eq:blocking_lb_basic} is known to be tight (for infinitely many values of $q, k$).
In fact, it is a major open problem to determine tight lower bounds on $b_q(k, s)$, for any $s > 1$.
Some special cases of this problem have been studied extensively under different names.  
A subset $B\subseteq \F_q^k$ is $s$-blocking if and only if the set $\F_q^k \setminus B$ does not contain a $(k-s)$-dimensional affine subspace. 
Hence, it follows from the density Hales-Jewett theorem~\cite{furstenberg1991density} that for fixed $q$, $d$, and $k \rightarrow \infty$, $b_q(k,k-d)=q^k-o(q^k)$.
In the special case $d=1$ and $q=3$, $1$-blocking sets are the complements of affine caps, and thus determining $b_3(k, k-1)$ is equivalent to the famous \textit{cap set problem}. 
The upper bound on affine caps proved in the work of Ellenberg and Gijswijt~\cite{ellenberg2017large}, which also uses the polynomial method, implies that $b_3(k, k - 1) > 3^k - 2.756^k$. 
Therefore, the lower bound of $b_3(k, k - 1) \geq 3^k - 3^{k - 1} - 1$ from \eqref{eq:blocking_lb_basic} is far from the truth. 
More generally, for $q=2,3$ we have $b_q(k,k-d)\geq q^k-o(c^k)$ for a constant $c<q$ depending on $q$ and $d$. For $q=2$, this is implicit in~\cite{bonin2000size} and for $q=3$ this follows from the multidimensional cap set theorem~\cite{fox2021popular} (see~\cite{gijswijt2021excluding} for a short proof for $q=2,3$). 
The exact asymptotics $b_2(k,k-2)=2^k-\Theta(2^{k/2})$ follows from~\cite{tait2021improved}.

As far as upper bounds on $b_q(k, s)$ are concerned, the general upper bound on vertex cover numbers in terms of fractional vertex cover numbers~\cite{lovasz1975ratio} implies that 
\begin{equation}\label{eq:blocking_up_basic}
b_q(k, s) \leq q^s\left(1+\ln \qbin{k}{s}_q \right),
\end{equation}
which can be upper bounded by $q^s(s(k - s) \ln q + 3)$ using the inequality 
$\qbin{k}{s}_q\leq e^2q^{s(k-s)}$ (see Lemma~\ref{lem:qbinestimate} for more precise estimates).

In this paper, we first prove the following bound which, for fixed $q$ and $s$, improves on \eqref{eq:blocking_up_basic} for $k$ large enough.

\begin{thm}\label{thm:blocking_upper}
Let $s, k$ be integers such that $2 \leq s \leq k$ and let $q$ be a prime power.
If $q = 2$, then 
\[b_q(k, s) \leq \frac{s(k - s) + s + 2}{\log_q{\frac{q^s}{q^s - 1}}} + 1.\]
If $q \geq 3$, then
\[b_q(k,s) \leq  (q^s-1)\cdot \frac{s(k-s)+s+2}{\log_q \frac{q^4}{q^3-q+1}}+1.\]
\end{thm}

The bound for $q \geq 3$ that we prove is in fact valid for $q = 2$ as well, but it is worse than the other bound, which is why we have separated the two cases.

We then focus on a particular notion of projective blocking sets and show that it is deeply connected to affine blocking sets. 
A \textit{strong $t$-blocking set} is a set of points in a projective space that intersects ever codimension-$t$ subspace in a set that spans the subspace. 
For $t = 1$, they are simply known as \textit{strong blocking sets}~\cite{davydov2011linear, heger2021short}. 
Recently, these objects have been shown to be equivalent to minimal codes from coding theory~\cite{ABNR2022}.
Minimal codewords in a linear code were first studied in 1980's for decoding purposes and then for their connection to cryptography (see~\cite{chabanne2013towards} and the references therein).
This ultimately led to the study of minimal codes: linear codes where \textit{every} codeword is minimal.
 Over the binary field, minimal codes are also equivalent to linear intersecting codes~\cite{cohen1985linear, cohen1994intersecting}, which are codes with the property that the supports of any two codewords have non-empty intersection.
 
The newfound equivalence between minimal codes and strong blocking sets has immensely increased the interest in proving bounds on the smallest size of a strong blocking sets and finding explicit constructions~\cite{abdn2022, alfarano2022geometric, ABN2023, bartoli2021small, heger2021short}.
We prove the following new equivalence between strong blocking sets and certain affine $2$-blocking sets, and use it to improve the best lower and upper bounds (for $q \geq 3$) on the smallest size of a strong blocking set. 
Let $\mathrm{PG}(k - 1, q)$ denote the $(k - 1)$-dimensional projective space obtained from the vector space $\mathbb{F}_q^k$. 
The points of $\mathrm{PG}(k - 1, q)$ correspond to lines passing through the origin in $\mathbb{F}_q^k$ (see Section~\ref{sec:prelim} for further details). 
Therefore, for every set $\mathcal{L}$ of points in $\mathrm{PG}(k - 1, q)$, we can construct the set $B = \cup_{\ell \in \mathcal{L}} \ell$ of points in $\mathbb{F}_q^k$. 
This allows us to translate the properties of $\mathcal{L}$ in the projective space to properties of the set $B$ in the affine space.

\begin{lemma}
\label{lem:equivalence_strong_affine}
    Let $\mathcal{L}$ be a set of points in $\mathrm{PG}(k - 1, q)$.
    Then $\mathcal{L}$ is a strong $(s - 1)$-blocking set if and only if the set $B = \cup_{\ell \in \mathcal{L}} \ell \subseteq \mathbb{F}_q^k$ is an affine $s$-blocking set.  
\end{lemma}

Let $b^*_q(k, t)$ denote the minimum size of a strong $t$-blocking set in $\mathrm{PG}(k - 1, q)$. 
From this equivalence, and the upper bound on $b_q(k, 2)$ given by Theorem~\ref{thm:blocking_upper}, we derive the following upper bound on $b^*_q(k, 1)$, which improves the previous best upper bound (see~\cite[Theorem 1.5]{heger2021short}) of
\[b^*_q(k, 1) \leq (q + 1) \frac{2(k - 1)}{1 + \frac{1}{(q + 1)^2 \ln q}}\]
for all $q \geq 3$ and $k$ sufficiently large. 

\begin{thm}
\label{thm:strong_upper}
The minimum size $b^*_q(k, 1)$ of a strong blocking set $\mathrm{PG}(k - 1, q)$ satisfies
\[ b^*_q(k, 1) \leq (q + 1)\frac{2k}{\log_q(\frac{q^4}{q^3-q+1})}.\]
\end{thm}

Using a mix of coding theoretic and geometric arguments, we then prove the following new lower bound on the size of a strong blocking set. 
\begin{thm}\label{thm:strong_lower}
For any prime power $q$, there is a constant $c_q>1$ such that every strong blocking set in $\mathrm{PG}(k - 1, q)$ has size at least 
$(c_q-o(1))(q+1)(k-1)$, where $o(1)$ only depends on $k$. 
\end{thm}
The constant $c_q$ can be taken to be the unique solution $x\geq 1$ to the equation 
\[M_q \left(\frac{q - 1}{x(q+1)}\right) = \frac{1}{x(q + 1)},\]
where $M_q$ is the function appearing in the MRRW bound for linear codes (see Theorem~\ref{thm:MRRW} below). 
We do not have a closed formula for $c_q$, but for any $q$ it can be computed efficiently up to an arbitrary order of precision using a computer. 
Some estimates on $c_q$ have been obtained in \cite{scotti2023lower}. 
For every $q \geq 3$, our result improves the previous best lower bound of $(q + 1)(k - 1)$~\cite{ABNR2022}.
For $q = 2$, our result matches the current best lower bound, which can be deduced from~\cite{katona1983minimal}.

\subsection{Linear trifferent codes}
Our study of affine $2$-blocking sets and strong blocking sets is mainly motivated by a new connection to the trifference problem, which we establish in this paper. 
A \textit{perfect $q$-hash} code of length $n$ is a subset~$C$ of $\{0, 1, \dots, q - 1\}^n$ such that for any $q$ distinct elements in $C$, there is a coordinate where they have pairwise distinct values.
Understanding the largest possible size of a perfect $q$-hash code is a natural extremal problem that has gained much attention since the 1980s because of its connections to various topics in cryptography, information theory, and computer science~\cite{guruswami2022beating, km88, wx2001, xy2021}. 
We will focus on the $q = 3$ case where these codes are also known as \textit{trifferent codes}, and the problem of determining their largest possible size is called the \textit{trifference problem}. 
Let $T(n)$ denote the largest size of a trifferent code of length~$n$. 
The exact value of $T(n)$ is only known for $n$ up to $10$, where the last six values were obtained very recently via computer searches~\cite{Fiore2022, Kurz2024}. 
Asymptotically, the upper bound 
\begin{equation}\label{eq:trifference_upper_basic}
T(n) \leq 2 (3/2)^n  
\end{equation}
obtained by K\"orner~\cite{Korner1973} in 1973 is still the best known upper bound, despite considerable effort (see for example~\cite{costa2021gap} where it is shown that a direct application of the slice rank method will not improve the bound.)
Similarly, the current best lower bound 
\begin{equation}\label{eq:trifference_lower_basic}
T(n) \geq (9/5)^{n/4}   
\end{equation}
was proved by K\"orner and Marton~\cite{km88} in 1988, who used a ``probabilistic lifting'' of the optimal trifferent code of length~$4$. 
A natural restriction of the trifference problem is to study \textit{linear} trifferent codes, that is, trifferent codes $C$ in $\F_3^n$ which are also vector subspaces. This restriction is motivated by the fact that the best known explicit constructions of trifferent codes are linear~\cite{wx2001}. 
Moreover, the probabilistic construction of K\"orner and Marton~\cite{km88} uses an optimal linear trifferent code in $\F_3^4$. 
Let $T_L(n)$ denote the largest size of a linear trifferent code of length~$n$. 
Pohoata and Zakharov~\cite{Pohoata2022} have recently proven the following upper bound on $T_L(n)$, which shows a big separation from the known upper bounds on $T(n)$: 
\begin{equation}\label{eq:pohoata}
    T_L(n) \leq 3^{(1/4 - \epsilon)n}.
\end{equation}
The $\epsilon$ in their result is a small positive number, not determined explicitly.
We prove an equivalence between linear trifferent codes and affine $2$-blocking sets in $\F_3^k$ that are a union of lines through the origin.
By Lemma~\ref{lem:equivalence_strong_affine}, this also implies an equivalence between strong blocking sets over $\F_3$ and linear trifferent codes. 
In fact, we show that a linear code $C$ is trifferent if and only if it is minimal. 
By using our new lower bounds on strong blocking sets, we deduce the following improvement to \eqref{eq:pohoata} and prove a lower bound on $T_L(n)$ by using Theorem~\ref{thm:strong_upper}. 

\begin{thm}
\label{thm:trifferent_bounds}
For $n$ large enough, the largest size of linear trifferent code of lenght $n$ has the following bounds:
\[\tfrac{1}{3} (9/5)^{n/4} \leq T_L(n) \leq 3^{n/4.55}.\]
\end{thm}

In \cite{wx2001} it was shown that $T_L(n) \geq 3^{n/112}$, via an explicit construction, whereas our bound is roughly $3^{n/7.48}$. 
Note that our lower bound on $T_L(n)$ is only a factor of $3$ away from the best lower bound $T(n) \geq (9/5)^{n/4}$ given in~\eqref{eq:trifference_lower_basic}.
Moreover, the lower bound on $T(n)$ was obtained by constructing non-linear trifferent codes~\cite{km88}, whereas we have constructed linear trifferent codes.\footnote{It has been pointed out to us by the referee that the idea of K\"orner and Marton can also be used to get a construction of linear trifferent codes that have size at least $(9/5)^{n/4}$.}
Even a tiny further improvement will break the current best lower bounds for the trifference problem, which have not been improved since 1988. 
Therefore, our results give a new motivation for studying linear trifferent codes and strong blocking sets. 

Our lower bounds on trifferent codes so far, which follow from our upper bounds on affine $2$-blocking sets, are based on probabilistic constructions. 
It is of great interest to also obtain explicit constructions (see for example~\cite{wx2001} and the references therein). 
In this direction, we first provide an explicit construction of affine $2$-blocking sets whose sizes are just a constant factor away from the best lower bound (given in \eqref{eq:blocking_lb_basic}).

\begin{thm}
\label{thm:explicit}
There is an absolute constant $c$, such that for every prime power $q$, and $k$ large enough, we can explicitly construct $c (q + 1) k$ lines through the origin in $\mathbb{F}_q^k$ whose union blocks every codimension-$2$ affine subspace, thus implying 
\[b_q(k, 2) \leq c (q^2 - 1) k + 1.\]
\end{thm}

Our construction is based on a recent breakthrough on explicit constructions of strong blocking sets~\cite{abdn2022}.
By restricting to $q = 3$, and using a different construction of strong blocking sets~\cite{bartoli2021small}, 
 we obtain a new explicit construction of linear trifferent codes, which improves the current best explicit construction of length-$n$ linear trifferent codes from dimension $n/112$ (obtained in ~\cite{wx2001}) to dimension $n/48$.
 By using a linear trifferent code of dimension $6$ in $\mathbb{F}_3^{24}$ and standard concatenation with algebraic geometric codes, we prove the following. 

\begin{thm}
\label{thm:explicit_trifferent}
    There exists an infinite sequence of lengths $n$ for which there is an explicit construction of linear trifferent codes of dimension $\lceil 23n/312 \rceil$. 
\end{thm}

\subsection*{Outline of the paper}
In Section \ref{sec:prelim}, we describe some basic theory of finite geometry and error-correcting codes. 
We also recall previous results on strong blocking sets and prove Lemma~\ref{lem:equivalence_strong_affine}. 
In Section \ref{sec:upper}, we prove Theorem~\ref{thm:blocking_upper}. We use this bound for the case $s=2$, along with Lemma~\ref{lem:equivalence_strong_affine}, to prove Theorem~\ref{thm:strong_upper}. 
In Section~\ref{sec:lower}, we prove Theorem~\ref{thm:strong_lower} using the MRRW bound from coding theory and lower bounds on affine blocking sets.  
In Section~\ref{sec:explicit}, we prove Theorem~\ref{thm:explicit} using the new explicit construction of strong blocking sets. 
Finally, in Section~\ref{sec:trifferent}, we focus on linear trifferent codes and prove Theorems~\ref{thm:trifferent_bounds} and \ref{thm:explicit_trifferent}. 
The diagram in Figure~\ref{fig:equiv} summarises all the equivalences between blocking sets, minimal codes and trifferent codes, which form the backbone of our work. 
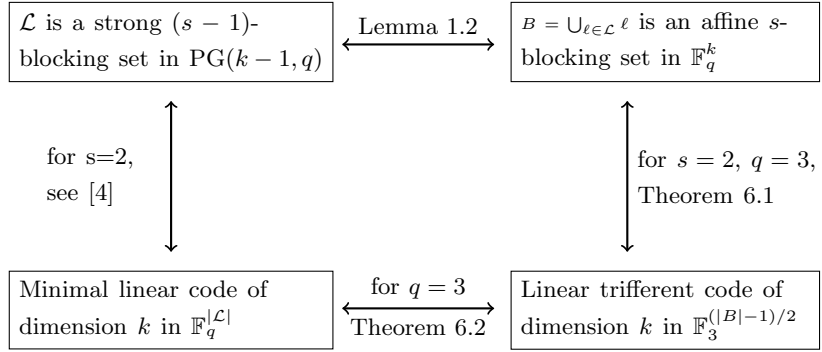
\begin{figure}[h!] 
\label{fig:equiv}
\centering
\begin{tikzpicture}[scale = 0.5]
\node[draw, text width=4cm,  above] at (0,0.5) {\footnotesize $\mathcal{L}$ is a strong $(s-1)$-blocking set in PG$(k-1,q)$};
\draw [thick, <->] (0,0)   --(0,-4);
\node[left, text width=1.5cm] at (0, -2) {\footnotesize for s=2,  see \cite{ABNR2022}};

\node[draw, text width=4cm, below] at (0,-4.5) {\footnotesize Minimal linear code of dimension $k$ in $\F_q^{|\mathcal{L}|}$};
\draw [thick, <->] (4.5,1.5)   --(8.5,1.5);
\node[above] at (6.5,1.5) {\footnotesize Lemma~\ref{lem:equivalence_strong_affine}};

\node[draw, text width=3.8cm, below] at (13,-4.5) {\footnotesize Linear trifferent code of dimension $k$ in $\mathbb{F}_3^{(|B|-1)/2}$ };
\draw [thick, <->] (4.5,-5.5)   --(8.5,-5.5);
\node[above] at (6.5,-5.5) {\footnotesize for $q = 3$};
\node[below] at (6.5,-5.5) {\footnotesize Theorem~\ref{thm:equivalence_triff_min}};

\node[draw, text width=3.8cm, above] at (13,0.5) {\footnotesize {\tiny $B = \bigcup_{\ell\in \mathcal{L}} \ell$} is an affine $s$-blocking set in $\mathbb{F}_q^k$};
\draw [thick, <->] (12,0)   --(12,-4);
\node[right, text width=2.4cm] at (12, -2) {\footnotesize  for $s=2$, $q=3$,  Theorem~\ref{thm:triff_to_block}};
\end{tikzpicture}
\caption{Equivalences between blocking sets and codes.}
\end{figure}

\section{Preliminaries}
\label{sec:prelim}

Throughout this paper we will use $q$ to denote a prime power. 
A projective space of dimension $k - 1$ over the finite field $\F_q$ is defined as 
\[\mathrm{PG}(k - 1, q) \coloneqq \left(\F_q^{k} \setminus \{\vec{0}\} \right) / \sim,\]
where 
$u \sim v$ if $u = \lambda v$ for some nonzero $\lambda \in \F_q$. 
The equivalence class that a nonzero vector $v$ belongs to is denote by $[v]$. 
The $1$-dimensional, $2$-dimensional, $\dots$, $(k - 1)$-dimensional vector subspaces of $\F_q^{k}$ correspond to the points, lines, $\dots$, hyperplanes of $\mathrm{PG}(k - 1, q)$. 
In fact, we will often identify the points of $\mathrm{PG}(k - 1, q)$ with the $1$-dimensional vector subspaces of $\F_q^k$, that is, lines of the affine space $\F_q^k$ that pass through the origin. 
Note that while the dimension of a projective subspace is one less than the dimension of its corresponding vector subspace, the codimensions remain the same. 
For a subset $S$ of points in $\mathrm{PG}(k - 1, q)$, the subspace formed by taking the linear span of $S$ is denoted by~$\langle S \rangle$. 
We refer to~\cite{casse2011} for further details on projective spaces.  

\begin{mydef}[$q$-binomial coefficient]
For integers $0 \leq s \leq k$, the $q$-binomial coefficient is defined as 
\[\qbin{k}{s}_q = \frac{(q^k - 1) \cdots (q^{k - s+ 1} - 1)}{(q^s - 1) \cdots (q - 1)},\]
where we define the empty product to be $1$. For other values of $s$ and $k$, we define $\qbin{k}{s}_q$ to be zero.
\end{mydef}

The number $\qbin{k}{s}_q$ is also known as the Gaussian coefficient. We will often use the fact that $\qbin{k}{s}_q=\qbin{k}{k-s}_q$. A straightforward double counting argument shows the following. 

\begin{lemma}\label{lem:sub_count}
\mbox{}
\begin{itemize}
    \item[\textup{(a)}] The number of $(s-1)$-dimensional subspaces of $\mathrm{PG}(k-1, q)$ is equal to $\qbin{k}{s}_q$.

    \item[\textup{(b)}] The number of $s$-dimensional affine subspaces of $\mathbb{F}_q^k$ is equal to $q^{k - s}\qbin{k}{s}_q$.

\end{itemize}
\end{lemma}

We will use the following estimates on Gaussian coefficients. 
\begin{lemma}\label{lem:qbinestimate} Let $q$ be a prime power and let $1\leq s\leq k$ be integers.
\begin{itemize}
\item[\textup{(a)}] We have:
\[
1\leq q^{-s(k-s)}\qbin{k}{s}_q \leq \frac{q}{q-1}e^{\frac{q}{(q^2 - 1)(q-1)}}.
\]
\item[\textup{(b)}] If $s\geq 3$ and $k\geq s+3$, we have:
\[
\frac{q^3}{(q^2-1)(q-1)}\leq q^{-s(k-s)}\qbin{k}{s}_q. 
\]
\end{itemize}
\end{lemma}
Note that the upper bound in (a) is decreasing in $q$ and is less than $2$ for $q\geq3$. 
\begin{proof}
We first prove part (a). Since $\frac{q^n-1}{q^m-1}\geq \frac{q^n}{q^m}$ for any positive integers $n\geq m$,  we have 
\[
\qbin{k}{s}_q=\prod_{i=0}^{s-1}\frac{q^{k-i}-1}{q^{s-i}-1}\geq \prod_{i=0}^{s-1}\frac{q^{k-i}}{q^{s-i}}=q^{s(k-s)}.
\]
For the other inequality, we use
\[
\qbin{k}{s}_q=\prod_{i=0}^{s-1}\frac{q^{k-i}-1}{q^{s-i}-1}\leq \prod_{i=0}^{s-1}\frac{q^{k-i}}{q^{s-i}-1} = q^{s(k-s)}\cdot \prod_{i=1}^s \frac{q^i}{q^i-1}.
\]
Since 
\[
\prod_{i=1}^s \frac{q^i}{q^i-1}\leq \prod_{i=1}^\infty \frac{q^i}{q^i-1}=\frac{q}{q-1}\prod_{i=2}^\infty (1+\frac{1}{q^i-1})\leq \frac{q}{q-1}e^{\sum_{i=2}^\infty\frac{1}{q^i-1}}\leq \frac{q}{q-1}e^{\frac{q}{(q^2-1)(q-1)}},
\]
where we used that $\sum_{i=2}^\infty\frac{1}{q^i-1}\leq \sum_{i=2}^\infty \frac{1}{q^i-q^{i-2}}=\frac{1}{q^2-1}\frac{q}{q-1}$, 
the inequality follows for all prime powers~$q$. 

For part (b), it can be easily verified that  
\[
\frac{(q^k-1)(q^{k-1}-1)(q^{k-2}-1)}{q^3-1}\geq \frac{q^kq^{k-1}q^{k-2}}{q^3}
\]
since $k\geq 6$. It follows that 
\begin{eqnarray*}
\qbin{k}{s}_q&=&\frac{(q^k-1)(q^{k-1}-1)(q^{k-2}-1)}{(q^3-1)(q^2-1)(q-1)}\cdot \frac{q^{k-3}-1}{q^s-1}\cdots \frac{q^{k-s+1}-1}{q^4-1}\\
&\geq& \frac{q^3}{(q^2-1)(q-1)}q^{3k-9}\cdot q^{(k-s-3)(s-3)}\\
&=& \frac{q^3}{(q^2-1)(q-1)}q^{(k-s)s},
\end{eqnarray*}
where in the inequality, we used that $k-3\geq s$.
\end{proof}

\begin{lemma}\label{lem:count_hyp}
Let $H$ be an affine hyperplane in $\mathbb{F}_q^k$ that does not pass through the origin. 
The number of $i$-dimensional affine subspaces disjoint from $H$ and passing through the origin is equal to 
$\qbin{k - 1}{i}_q$. 
\end{lemma}
\begin{proof}
Any such subspace must be contained in the unique hyperplane parallel to $H$ that passes through the origin, and by Lemma~\ref{lem:sub_count} the number of $i$-dimensional vector subspaces of a $(k-1)$-dimensional vector space is equal to $\qbin{k - 1}{i}_q$.
\end{proof}

\begin{mydef}
\label{def:nqks}
Let $H$ be a $(k - s)$-dimensional affine subspace of $\F_q^k$, not passing through origin. We denote by $n_q(k,s)$ the number of $s$-dimensional affine subspaces through the origin that are disjoint from $H$.
\end{mydef}

Note that $n_q(k,s)$ is independent of the particular choice of $H$ since the general linear group acts transitively on $(k-s)$-dimensional not passing through the origin. A formula is given by the following lemma.

\begin{lemma}\label{lem:countnrdisjoint}
We have  
\[ n_q(k,s)=\sum_{i=1}^{s} q^{(s-i)(k-i-s+1)}\qbin{s-1}{i-1}_q\qbin{k-s}{i}_q. \]
\end{lemma}

\begin{proof}
Let $H$ be a $(k - s)$-dimensional affine subspace of $\F_q^k$ not passing through origin and let $H'$ be the $(k - s + 1)$-dimensional subspace spanned by $H$ and the origin.
For $1 \leq i \leq s$, let $\mathcal{S}_i$ denote the set of $s$-dimensional subspaces through the origin, disjoint from $H$, that intersect $H'$ in an $i$-dimensional subspace.
As $\mathcal{S}_1 \cup \cdots \cup \mathcal{S}_s$ is a partition of the set of subspaces that we need to count, we have
$n_q(k, s) = \sum_{i = 1}^s |\mathcal{S}_i|$. 
Therefore, it suffices to show that 
$|\mathcal{S}_i| = q^{(s-i)(k-i-s+1)}\qbin{s-1}{i-1}_q\qbin{k-s}{i}_q$.
The number of ways of picking an $i$-dimensional subspace in $H'$ that is disjoint from $H$ is equal to 
$\qbin{k - s}{i}_q$ by Lemma~\ref{lem:count_hyp}. 
Once we have picked such a subspace $T$, the number of ways of picking an $s$-dimensional subspace $S$ such that $S \cap H' = T$ amounts to picking $s - i$ remaining basis vectors outside $H'$, which gives us exactly 
\[\frac{(q^{k} - q^{k - s + 1}) \cdots (q^k - q^{k - i})}{(q^s - q^i) \cdots (q^s - q^{s - 1})} =q^{(s-i)(k-i-s+1)}\qbin{s-1}{i-1}_q
\]
choices for $S$. 
\end{proof}

\subsection{Error-correcting codes}

We now recall some definitions and results from coding theory (see~\cite{coding_theory_book} for a standard reference). 

\begin{mydef}
 The support of a vector $v\in \Fq^n$ is the set 
 \[ 
 \supp(v):=\{ i \, : \, v_i\neq 0\} \subseteq [n].
 \]
 The Hamming weight of $v$ is
 \[
 \wt(v):=|\supp(v)|.
 \]
\end{mydef}

The Hamming weight induces a metric on $\Fq^n$, given by $d(u,v):=\wt(u-v)$, which is known as the \textit{Hamming distance}. 

\begin{mydef}
An $[n,k,d]_q$ code $C$ is a $k$-dimensional subspace of $\Fq^n$, with \textit{minimum distance}
\[
d:=\min\{\wt(v) \, : \, v \in C\setminus\{\vec{0}\}\}.
\] 
The elements of $C$ are called \textit{codewords}. 
A \textit{generator matrix} for $C$ is a matrix $G\in \Fq^{k\times n}$ such that 
\[
C=\{uG\,:\, u \in \Fq^k\}.
\]
The rate of $C$ is equal to $k/n$ and the relative Hamming distance of $C$ is equal to $d/n$. 
\end{mydef}

\begin{mydef}
Let $\{n_i \}_{i \geq 1}$ be an increasing sequence of lengths and suppose that there exist sequences $\{k_i\}_{i \geq 1}$ and $\{d_i\}_{i \geq 1}$ such that
for all $i \geq 1$ there exists an $[n_i, k_i, d_i]_q$ code $C_i$. 
Then the sequence $\C = \{C_i\}_{i \geq 1}$ is called a family of codes. 
The rate of $\C$ is defined as 
\[R(\C) = \lim_{i \rightarrow \infty} \frac{k_i}{n_i},\]
and the relative distance of $\C$ is defined as 
\[\delta(\C) = \lim_{i \rightarrow \infty} \frac{d_i}{n_i},\]
assuming that these limits exist.
\end{mydef}

A family $\C$ for which $R(\C) > 0$ and $\delta(\C) > 0$, is known as an \textit{asymptotically good code}, and various explicit constructions of such codes are known~\cite{coding_theory_book, tsfasman2013algebraic}. However, the problem of understanding the optimal trade-off between the rate and relative distance of a family of codes is in general not well understood. As a result, the following open question lies at the heart of the subject:
\begin{quote}
What is the \textit{largest} rate that can be achieved for a family of codes of a given relative distance $\delta$?
\end{quote}
For our purposes, we will focus on upper bounds on the rate. To continue, let us define the following two functions.
The $q$-ary entropy is defined by:
\[
\ent_q(x) := x \log_q(q-1) - x\log_q x - (1-x)\log_q(1-x).
\]
Define
\[
M_q(\delta) := \ent_q\left(\frac{1}{q}\left(q - 1 - (q-2)\delta - 2\sqrt{(q-1)\delta(1 - \delta)}\right)\right).
\]
A relevant property of $M_q$ is that it provides an asymptotic upper bound on the rate of any code with a given relative distance. 

\begin{thm}[MRRW bound for $q$-ary codes \cite{Aaltonen79, MRRW77}]\label{thm:MRRW}
For any fixed alphabet size $q$, and relative distance $\delta > 0$, any family of $q$-ary codes with relative distance $\delta$ and rate $R$ satisfies
\[
R \leq M_q(\delta).
\]
\end{thm}

One may verify that $M_q$ is a continuous and strictly decreasing function in the domain $[0, 1 - 1/q]$, satisfying $M_q(0) = 1$ and $M_q(1 - 1/q) = 0$, which can be used to show that the constant $c_q$ appearing in Theorem~\ref{thm:strong_lower} is well defined.
Note that $c_q$ is also equal to the maximum $x$ for which $M_q((q - 1)/(x(q+1)) \leq 1/(x(q+1))$ and $x \geq 1$ (see Appendix~\ref{sec:appendix} for further details).

\subsection{Strong blocking sets}

\begin{mydef}
 Let $C$ be an $[n,k,d]_q$ code. A nonzero codeword $v \in C$ is said to be \textit{minimal} if $\supp(v)$ is minimal with respect to inclusion in the set 
 \[
 \{ \supp(c) \, : \, u \in C\setminus\{\vec{0}\}\}.
 \]
 The code $C$ is a \textit{minimal (linear) code} if all its nonzero codewords are minimal.
\end{mydef}

So a linear code is minimal if for all $u,v\in C$ we have: $\supp(u)\subsetneq \supp(v) \implies u=0$.

\begin{mydef}
A set $S$ of points in a projective space is called a strong blocking set if for every hyperplane $H$, we have
$\langle S \cap H \rangle = H$. 
\end{mydef}

These special kinds of (projective) blocking sets have been studied under the names of generator sets~\cite{FS2014} and cutting blocking sets~\cite{ABNR2022}, but we adopt the terminology of strong blocking sets used in the earliest work~\cite{davydov2011linear}.

\begin{mydef}
An $[n,k,d]_q$ code $C$ is nondegenerate if there is no coordinate where every codeword has $0$ entry. 
\end{mydef}

The following is a standard equivalence between nondegenerate codes and certain sets of points in the projective space (see for example~\cite[Theorem 1.1.6]{tsfasman2013algebraic}).

\begin{lemma}\label{lem:intersection}
Let $C$ be a nondegenerate $[n,k,d]_q$ code and let $G=(g_1 \mid \ldots \mid g_n) \in \Fq^{k\times n}$ be any of its generator matrices. 
Let $\mathcal{P} = \{[g_1], \ldots, [g_n]\}$ be the (multi-)set of points in $\mathrm{PG}(k - 1, q)$ given by~$G$. 
Then 
\[d = n - \max_H \{|\{i : [g_i] \in H\} |\},\]
where the maximum is taken over all hyperplanes $H$.
Conversely for $d > 0$, any set of $n$ points in $\mathrm{PG}(k - 1, q)$ whose maximum intersection with a hyperplane has size $n - d$ gives rise to a nondegenrate $[n, k, d]_q$ code by taking a generator matrix whose columns are the coordinates of these $n$ points. 
\end{lemma}

In particular, Lemma~\ref{lem:intersection} implies that if a set of $n$ points in $\mathrm{PG}(k - 1, q)$ meets every hyperplane in at most $m$ points, then there exists an $[n, k, n - m]_q$ code. 

Under this correspondence, minimality of an $[n, k, d]_q$ code has been shown to be equivalent to the corresponding point set giving rise to a strong blocking set in $\mathrm{PG}(k - 1, q)$. 

\begin{thm}[see \cite{ABNR2022}, \cite{tang2021full}]
\label{thm:equivalence_minimal_strong}
 Let $C$ be a nondegenerate $[n,k,d]_q$ code and let $G=(g_1 \mid \ldots \mid g_n) \in \Fq^{k\times n}$ be any of its generator matrices. The following are equivalent:
 \begin{itemize}
     \item[\textup{(i)}] $C$ is a minimal code;
     \item[\textup{(ii)}] The set $\{[ g_1] ,\ldots, [g_n ]\}$ is a strong blocking set in $\mathrm{PG}(k-1,q)$.
 \end{itemize}
\end{thm}

\begin{remark}\label{rem:strong_distance}
    Note that the Hamming distance $d$ plays no role in the definition of minimal codes. 
    However, it can be shown that if an $[n, k, d]_q$ code is minimal, then $d \geq (q - 1)(k - 1) + 1$ (see~\cite[Theorem 2.8]{ABNR2022}).
    This implies that any minimal $[n, k, d]_q$ code where $k$ is a linear function of $n$ is asymptotically good, which is another motivation for studying these codes. 
\end{remark}

The main problem is to find small strong blocking sets in $\mathrm{PG}(k - 1, q)$, which is then equivalent to finding short minimal codes of dimension $k$. 
The following is the best lower bound on the size of a strong blocking set, for all $q \geq 3$, and it was proved using the polynomial method.

\begin{thm}[see Theorem 2.14 in \cite{ABNR2022}]
\label{thm:lower_bound_strong}
 Let $S \subseteq \mathrm{PG}(k-1,q)$ be a strong blocking set. Then
 \[
 |S|\geq (q+1)(k-1).
 \]
\end{thm}

\begin{remark}
    This is in general not tight. For $q = 2$, a better bound follows from~\cite{katona1983minimal} and Lemma~\ref{lem:equivalence_strong_affine} and for $k = 3$, a better lower bound follows from~\cite{ball1996multiple} as then a strong blocking set equivalent to a $2$-fold blocking set. 
\end{remark}

The following is the best upper bound on the smallest size of a strong blocking set. 
\begin{thm}[see \cite{heger2021short}]
\label{thm:upper_bound_strong}
 The the smallest size of a strong blocking set in $\mathrm{PG}(k-1,q)$ is at most
\[ \begin{cases}\frac{2k-1}{\log_2(4/3)} & \mbox{ if } q=2, \\
  (q+1)\left\lceil\frac{2}{1+\frac{1}{(q+1)^2\ln q}}(k-1)\right\rceil & \mbox{ otherwise. }
  \end{cases}
  \]
\end{thm}

\begin{remark}
For $q = 2$, this bound already appears in~\cite{miklos1984linear}, where it is attributed to Koml\'os.
\end{remark}

We now prove the new characterization of strong blocking sets based on affine blocking sets, outlined in Lemma~\ref{lem:equivalence_strong_affine}.
First, we define a generalization of strong blocking sets, which also appears in~\cite[Definition 2]{fancsali2016higgledy} under the name of generator sets. 

\begin{mydef}
    A set of points in $\mathrm{PG}(k - 1, q)$ is called a strong $t$-blocking set if it intersects every codimension-$t$ subspace in a set of points that spans the subspace. 
\end{mydef}

\begin{proof}[Proof of Lemma~\ref{lem:equivalence_strong_affine}]
First assume that $\mathcal{L}$ is a strong $(s-1)$-blocking set in PG$(k-1,q)$. 
Let $V\subseteq \F_q^k$ be a codimension-$s$ vector subspace and let $u\in \F_q^k\setminus V$. It suffices to show that $B \cap (V+u) \neq \emptyset$. 
Let $W = \langle V, u \rangle = \cup_{\lambda \in \mathbb{F}_q} (V + \lambda u)$ be the codimension-$(s-1)$ vector subspace spanned by $V\cup\{u\}$.
Then $W$ meets $\mathcal{L}$ in a spanning set. 
In particular, there exists an element $b \in B$ such that $b \in W \setminus V$. 
Thus we can write $b = v + \lambda u$ for some $v \in V$ and $\lambda \neq 0$. 
Then $b' = \lambda^{-1} b$ is contained in $B$ as $B$ is closed under taking scalar multiples, and $b'$ is contained in $V + u$ as $b' = \lambda^{-1} v + u$ and $\lambda^{-1}v \in V$. 

Now assume that $B$ blocks all codimension-$s$ affine subspaces. 
Say $\mathcal{L}$ is not a strong $(s - 1)$-blocking set. 
Then there is some codimension-$(s-1)$ vector subspace $V$ such that $\mathcal{L} \cap V$ is contained inside a codimension-$s$ vector subspace $V'$ of $V$. 
Then the set $B$ does not block any affine codimension-$s$ subspace contained in $V \setminus V'$ that is parallel to $V'$. 
\end{proof}

\begin{remark}
    By combining Lemma~\ref{lem:equivalence_strong_affine} and \eqref{eq:blocking_lb_basic}, we get a new proof of the lower bound on strong blocking sets given in Theorem~\ref{thm:lower_bound_strong} (see \cite{heger2021short} for another proof that uses \eqref{eq:blocking_lb_basic}). 
\end{remark}

\section{Upper bounds on blocking sets}
\label{sec:upper}

In this section, we give a probabilistic construction of $s$-blockings sets in $\F_q^k$, thus obtaining an upper bound on $b_q(k,s)$.
For $q = 2$, we simply pick random points and show that if the number of points is large enough, then the probability that they form an $s$-blocking set is positive. 
Our novel idea for $q \geq 3$ is to randomly pick $s$-dimensional linear subspaces instead. 
We start with a preliminary lemma that estimates $n_q(k,s)$ as defined in Definition~\ref{def:nqks}

\begin{lemma}\label{lem:grensqvdm} Let $q$ be a prime power. Let $2\leq s\leq k$ be integers. We have the following estimate on $n_q(k,s)$.
\[
n_q(k,s)\leq \frac{q^3-q+1}{q^4}\qbin{k}{s}_q.
\]
\end{lemma}
\begin{proof}
We will split the proof into three cases: $s=2$; $s\geq 3, k\geq s+3$; $s\geq 3, s\leq k\leq s+2$.
\paragraph{Case $s=2$.}
We use Lemma~\ref{lem:countnrdisjoint} and write out the given expression for $s=2$. We obtain
\begin{eqnarray*}
\frac{n_q(k,2)}{\qbin{k}{2}_q}&=&\frac{\qbin{1}{0}_q\qbin{k-2}{1}_qq^{k-2}+\qbin{1}{1}_q\qbin{k-2}{2}_q}{\qbin{k}{2}_q}\\
&\leq&\frac{(q^{k-2}-1)(q^k-q^{k-2}+q^{k-3}-1)}{(q^k-1)(q^{k-1}-1)}\\
&\leq &\frac{1}{q}\frac{q^k-q^{k-2}+q^{k-3}-1}{(q^k-1)}\\
\end{eqnarray*}
\begin{eqnarray*}
&=&\frac{1}{q}(1-\frac{q^{k-2}-q^{k-3}}{q^k-1})\\
&\leq &\frac{1}{q}(1-q^{-2}+q^{-3})\\
&=& \frac{q^3-q+1}{q^4}.
\end{eqnarray*}
The inequality in the second line is an equality if $k-2\geq 2$.

\paragraph{Case $s\geq 3$ and $k\geq s+3$.}
We will use the $q$-Vandermonde identity (see \cite{stanley2011enumerative}, solution of exercise 1.100):
\begin{equation*}
\qbin{m+n}{\ell}_q=\sum_{j}\qbin{n}{j}_q\qbin{m}{\ell-j}_qq^{(m-\ell+j)j}.
\end{equation*}
Substituting $\ell=s$, $m=k-s$, $n=s-1$ and $j=s-i$, we obtain
\begin{equation}\label{eq:qvandermonde}
\qbin{k-1}{s}_q=\sum_{i=1}^s \qbin{s-1}{s-i}_q\qbin{k-s}{i}_qq^{(k-s-i)(s-i)},
\end{equation}
which is the form we will use.

Using again Lemma~\ref{lem:countnrdisjoint}, we obtain
\begin{eqnarray*}
n_q(k,s)&=&\sum_{i=1}^{s} \qbin{s-1}{i-1}_q\qbin{k-s}{i}_q q^{(s-i)(k-i-s+1)}\\
&=& \sum_{i=1}^{s} \qbin{s-1}{s-i}_q\qbin{k-s}{i}_q q^{(k-s-i)(s-i)} q^{s-i}\\
&\leq & (q^{s-1}-q^{s-2})\qbin{k-s}{1}_q q^{(k-s-1)(s-1)} + q^{s - 2} \left(\sum_{i = 1}^s \qbin{s-1}{s-i}_q\qbin{k-s}{i}_q q^{(k-s-i)(s-i)} \right).
\end{eqnarray*}
The first summand can be upper bounded by
\[
(q^{s-1}-q^{s-2})\qbin{k-s}{1}_q q^{(k-s-1)(s-1)}\leq \frac{q^{(k-s)s}}{q}\leq \frac{(q^2-1)(q-1)}{q^4}\qbin{k}{s}_q,
\]
where we used Lemma~\ref{lem:qbinestimate}(b) in the second inequality.

The second summand can be upper bounded using the $q$-Vandermonde identity (\ref{eq:qvandermonde}):
\[
q^{s - 2} \left(\sum_{i = 1}^s \qbin{s-1}{s-i}_q\qbin{k-s}{i}_q q^{(k-s-i)(s-i)} \right) = q^{s-2}\qbin{k-1}{s}_q\leq \frac{1}{q^2}\qbin{k}{s}_q.
\]
The result now follows from combining the two bounds and the fact that $\frac{(q^2-1)(q-1)}{q^4}+\frac{1}{q^2}=\frac{q^3-q+1}{q^4}$.

\paragraph{Case $s\geq 3$ and $s\leq k\leq s+2$.}
If $k=s$, we have $n_q(k,s)=0\leq \frac{q^3-q+1}{q^4}\qbin{s}{s}_q$ as required.

If $k=s+1$, then by Lemma~\ref{lem:countnrdisjoint} we have $n_q(k,s)=q^{s-1}$ and
\[
\frac{q^3-q+1}{q^4}\qbin{k}{s}_q=\frac{q^{s+4}-q^{s+2}+q^{s+1}-q^3+q-1}{(q-1)q^4}.
\]
We leave it to the reader to verify that indeed $q^{s-1}(q-1)q^4\leq q^{s+4}-q^{s+2}+q^{s+1}-q^3+q-1$.

Finally, suppose that $k=s+2$. We have
\[
n_q(k,s)=q^{2(s-1)}\qbin{2}{1}_q+q^{s-2}\qbin{s-1}{1}_q=\frac{q^{2s}-q^{2s-2}+q^{2s-3}-q^{s-2}}{q-1}\leq\frac{q^3-q+1}{q-1}q^{2s-3}. 
\]
On the other hand, using that $(q^{s+2}-1)(q^{s+1}-1)q^2\geq q^{s+2}q^{s+1}(q^2-1)$, we have
\[
\qbin{k}{s}_q=\frac{(q^{s+2}-1)(q^{s+1}-1)}{(q^2-1)(q-1)}\geq \frac{q^{2s+1}}{q-1}.
\]
Combining the two inequalities, we obtain $\qbin{k}{s}_q\frac{q^3-q+1}{q^4}\geq n_q(k,s)$.
\end{proof}

\begin{proof}[Proof of Theorem~\ref{thm:blocking_upper}]
First let $q = 2$. 
Let $P_1, \dots, P_m$ be points in $\F_2^k$ chosen uniformly at random independently from each other. 
Let $H$ be a codimension-$s$ affine subspace. 
The probability that $H$ does not contain any $P_1, \dots, P_m$ is equal to 
$(1 - 2^{-s})^m$. 
Since there are in total $2^s\qbin{k}{k - s}_2 = 2^s \qbin{k}{s}_2$ choices for $H$, the probability that
$B = \{P_1, \dots, P_m\}$ is not an $s$-blocking set is at most 
\begin{align*}
    \left(\frac{2^s - 1}{2^s}\right)^m\cdot 2^s\qbin{k}{s}_2.
\end{align*}
By Lemma~\ref{lem:qbinestimate}, this probability is upper bounded by 
\begin{align*}
    \left(\frac{2^s - 1}{2^s}\right)^m 2^{s(k - s) + s + 1} e^{2/3} < 2^{- m \log_{2}\left(\frac{2^s}{2^s - 1}\right) + s(k-s) + s + 2}.
\end{align*}
It follows that the probability is less than $1$ for $m \geq \frac{s(k - s) + s + 2}{\log_2{\frac{2^s}{2^s - 1}}}$, proving that there exists a collection of $\lceil \frac{s(k - s) + s + 2}{\log_2{\frac{2^s}{2^s - 1}}}\rceil$ points blocking all codimension-$s$ affine subspaces in $\F_2^k$.

We now consider the case $q \geq 3$. 
Let \[m\geq -1+ \frac{s(k-s)+s+2}{\log_q \left(\frac{q^4}{q^3-q+1}\right)}\] be an integer and let $\pi_1, \dots, \pi_m$ be $s$-dimensional spaces through the origin chosen uniformly at random, and independently from each other. 

For any codim-$s$ affine subspace $H$ that does not pass through the origin, the probability that $\pi_1,\ldots, \pi_m$ are disjoint from $H$ is at most $(\frac{q^3-q+1}{q^4})^m$ by Lemma~\ref{lem:grensqvdm}. The number of such $H$ equals $(q^s-1)\qbin{k}{s}_q$. Hence, the expected number of $H$ that are disjoint from $\pi_1,\ldots, \pi_m$ is at most
\begin{eqnarray*}
\left(\frac{q^3-q+1}{q^4}\right)^m\cdot (q^s-1)\qbin{k}{s}_q &\leq &\frac{q^4}{q^3-q+1}\cdot q^{-s(k-s)-s-2}(q^s-1)\qbin{k}{s}_q\\
&<&\frac{2q^4}{q^3-q+1}\cdot q^{-2}\\
&\leq & \frac{2q^2}{q^3-q+1}\\
&\leq&1.
\end{eqnarray*}
Here we used $\qbin{k}{s}_q\leq 2q^{s(k-s)}$, see Lemma~\ref{lem:qbinestimate}(a).

So there is an instance where $B=\pi_1\cup\cdots\cup \pi_m$ is a $s$-blocking set. By taking $m$ as small as possible, we have $m\leq(s(k-s)+s+2)/\log_q(\frac{q^4}{q^3-q+1})$, so we obtain
\[
b_q(k,s)\leq (q^s-1)\frac{s(k-s)+s+2}{\log_q(\frac{q^4}{q^3-q+1})}+1\]
as required.
\end{proof}

\begin{remark}
\label{rem:best_random}
For a fixed $s$, picking random points instead of picking random $s$-dimensional subspaces through the origin gives the upper bound of $\sim (q^s \ln{q}) s (k-s)$, which is worse for every $q \geq 3$, since $\ln q > 1$ for $q > e$. 
It can also be shown that picking $i$-dimensional subspaces, for any $i < s$, gives us worse bounds for $q \geq 3$. 
\end{remark}

\begin{proof}[Proof of Theorem~\ref{thm:strong_upper}]
    Let $s = 2$ and $q \geq 3$. 
    In the proof of Theorem~\ref{thm:blocking_upper} above, we have a random collection of planes through the origin whose union blocks every codimension-$2$ affine subspace in $\mathbb{F}_q^k$. 
    Each plane consists of $q + 1$ lines through the origin, which implies that we have a set of 
    at most $(q + 1)2k/\log_q(\frac{q^4}{q^3-q+1})$ lines through the origin in $\mathbb{F}_q^k$ whose union blocks every codimension-$2$ affine subspace. 
    By Lemma~\ref{lem:equivalence_strong_affine}, this collection of lines corresponds to a strong blocking set in $\mathrm{PG}(k - 1, q)$ of the same size.
\end{proof}

\begin{remark}
Our upper bound improves the previous best upper bound on strong blocking sets (see Theorem~\ref{thm:upper_bound_strong}) for all $q \geq 3$. 
Independent of our work, Alfarano, Borello and Neri have obtained the same upper bound in~\cite{ABN2023} using different techniques. 
We have also been informed by the authors of \cite{heger2021short} that a more careful analysis of their argument implies our upper bound. 
\end{remark}


\begin{remark}
Our upper bound on affine $s$-blocking sets also implies that for $2\leq s\leq k$, the smallest size of a strong $(s - 1)$-blocking set in $\mathrm{PG}(k - 1, q)$, for $q \geq 3$, is at most
\[ \frac{q^s-1}{q-1}\cdot \frac{s(k-s)+s+2}{\log_q(\frac{q^4}{q^3-q+1})}. \]
\end{remark}

\section{Lower bounds on strong blocking sets}
\label{sec:lower}
In this section we prove Theorem~\ref{thm:strong_lower}. Let $q$ be a prime power.
\begin{lemma}\label{lem:strong_to_code} Let $k$ be a positive integer. Then there is a $[b_q^*(k,1),k,(q-1)(k-1)+1]_q$ code.
\end{lemma}
\begin{proof}
Let $\{[g_1],\ldots, [g_n]\}\subseteq \mathrm{PG}(k-1,q)$ be a strong blocking set of size $n:=b_q^*(k,1)$ and let $C$ be the nondegenerate code with generator matrix $G=(g_1\mid\ldots\mid g_n)$. By Theorem~\ref{thm:equivalence_minimal_strong}, $C$ is a minimal code, which has minimum distance $d\geq (q-1)(k-1)+1$ by Remark~\ref{rem:strong_distance}. 
\end{proof}

\begin{proof}[Proof of Theorem~\ref{thm:strong_lower}]
Recall that $c_q$ is the unique solution $x\geq 1$ to $M_q \left(\frac{q - 1}{x(q+1)}\right) = \frac{1}{x(q + 1)}$. In Appendix~\ref{sec:appendix} we show that $c_q$ is well-defined and that $c_q > 1$. 
Let $c$ be a constant such that $1<c<c_q$. We will show that $b_q^*(k,1)\geq c(q+1)(k-1)$ for $k$ large enough.

Set $R=\frac{1}{c(q+1)}$ and $\delta=\frac{q-1}{c(q+1)}$. Note that $\delta<1-\frac{1}{q}$. Since $c<c_q$ and $M_q$ is strictly decreasing on $[0,1-1/q]$, we have $M_q(\delta)<R$. 
By the MRRW bound for $q$-ary codes, Theorem~\ref{thm:MRRW}, there is an integer $k_0$ such that for all $k\geq k_0$ there is no linear code $C\subseteq \F_q^n$ with rate at least $R$ and relative minimum distance at least $\delta$. 

Now let $k\geq k_0$. By Lemma~\ref{lem:strong_to_code}, there is a code $C$ of block length $n=b_q^*(k,1)$, dimension $k$ and minimum distance $d\geq (q-1)(k-1)+1$. It follows that $b_q^*(k,1)\geq c(q+1)(k-1)$ since otherwise the code $C$ has rate $\tfrac{k}{n}> \tfrac{1}{c(q+1)}=R$ and relative distance 
\[
\frac{d}{n}> \frac{(q-1)(k-1)+1}{c(q+1)(k-1)}> \delta.
\]
This concludes the proof. 
\end{proof}

\begin{remark}
From Lemma~\ref{lem:equivalence_strong_affine} and \eqref{eq:blocking_lb_basic}, it follows that every strong $(s - 1)$-blocking set in $\mathrm{PG}(k - 1, q)$ has size at least $\frac{q^s - 1}{q - 1} (k - s + 1)$. 
Using an inductive argument, we can also improve this lower bound by a constant factor, for every fixed $q, s$ and large $k$. 
\end{remark}

\section{Explicit constructions} 
\label{sec:explicit}

In Section~\ref{sec:upper} we constructed $s$-blocking sets by picking random $s$-dimensional subspaces through the origin in $\mathbb{F}_q^k$.
However, for explicit constructions, we will pick $1$-dimensional subspaces as then we can use the connection to strong $(s - 1)$-blocking sets in PG$(k-1,q)$ outlined in Lemma~\ref{lem:equivalence_strong_affine}. 
If the strong $(s - 1)$-blocking set has size $m$, then the corresponding affine $s$-blocking set has size $(q - 1) m + 1$. 
For example, if $q > s(k - s)$, there exists an explicit construction of strong $(s - 1)$-blocking sets of size at most $(s(k - s) + 1) (q^s - 1)/(q - 1)$ in $\mathrm{PG}(k - 1, q)$~\cite{fancsali2016higgledy, guruswami2016explicit}, which we can use to give an explicit construction of affine $s$-blocking sets of size at most $(q^s - 1) s(k-s) + q^s$. 
While this is a good explicit construction, it requires the field to be large with respect to $k$ and $s$. 
We will focus on fixed $q, s$, and large $k$. 

The main focus of explicit constructions has been on the special case of strong $1$-blocking sets, as these objects are equivalent to minimal codes~\cite{ABNR2022}. 
An easy construction, known as the ``tetrahedron'', is as follows. 
Take the union of all lines joining pairs of $k$ points in general position to get a strong blocking set of size $\binom{k}{2}(q - 1) + k$ in $\mathrm{PG}(k - 1, q)$, and thus an affine $2$-blocking set of size $\binom{k}{2}(q - 1)^2 + k(q-1) + 1$ in $\mathbb{F}_q^k$. 
Note that the dependency on the dimension $k$ is quadratic in this construction.
For $q = 2$, minimal codes are equivalent to intersecting codes, and thus we already have explicit constructions of strong blocking sets in $\mathbb{F}_2^k$ of size linear in $k$~\cite{cohen1994intersecting}. 
 Recently, an explicit construction of a strong blocking set of size linear in the dimension, for any fixed $q \geq 3$, was obtained by Bartoli and Borello~\cite[Corollary 3.3]{bartoli2021small}.
 The same construction also appears in an earlier work of Cohen, Mesnager and Randriam \cite{cohen2016yet}, and the main idea is to concatenate algebraic geometric codes with the simplex code. 
 They proved that for every prime power $q$, there exists an infinite sequence of $k$'s such that there is an explicit construction of a strong blocking set in the projective space $\mathrm{PG}(k - 1, q)$ of size $\sim q^4 k/4$.
 The problem of giving an explicit construction which also has a linear dependence on $q$, has recently been solved in~\cite{abdn2022}.

 \begin{thm}[Alon, Bishnoi, Das, Neri 2023]
 \label{thm:abdn}
There is an absolute constant $c$ such that for every prime power $q$ and $k$ large enough, there is an explicit construction of strong blocking sets in $\mathrm{PG}(k - 1, q)$ of size at most $c(q + 1)k$. 
 \end{thm}

\begin{proof}[Proof of Theorem~\ref{thm:explicit}]
Let $S$ be an explicitly constructed strong blocking set of size at most $c (q + 1) k$ in $\mathrm{PG}(k - 1, q)$. 
By Lemma~\ref{lem:equivalence_strong_affine}, the union of lines in $\mathbb{F}_q^k$ corresponding to the points of $S$ gives us an explicit construction of 
an affine-$2$ blocking set in $\mathbb{F}_q^k$ of size $(q - 1) c (q + 1) k + 1 = c (q^2 - 1) k + 1$. 
\end{proof}

For the sake of completeness we give a sketch of the construction in~\cite{abdn2022}. 
\begin{lemma}\label{lem:main_const}
 Let $\mathcal{M}=\{P_1,\ldots,P_n\}$ be a set of points in $\mathrm{PG}(k-1,q)$ and let $G=(\mathcal{M},E)$ be a graph on these points.
 If for every $S\subseteq \mathcal{M}$ there exists a connected component $C$ in $G - S$ such that 
 \[\langle S\cup C\rangle=\mathrm{PG}(k-1,q),\]
 then the set \[\bigcup_{ P_iP_j \in E} \langle P_i, P_j \rangle\] 
 is a strong blocking set. 
 \end{lemma}

For a graph $G$, its vertex integrity~\cite{bagga1992survey}  is defined as 
\[\iota(G) = \min_{S \subseteq V(G)} \left( |S| +  \kappa(G - S)\right),\]
where $\kappa(G - S)$ denotes the size of the largest connected component of the graph obtained by deleting the set $S$ of vertices from $G$. 
Graphs with high vertex integrity along with points in $\mathrm{PG}(k - 1, q)$ that have low intersection with every hyperplane give rise to the construction that we want.

\begin{cor}
Let $\mathcal{P}=\{P_1,\ldots,P_n\}$ be a set of points in $\mathrm{PG}(k-1,q)$ such that every hyperplane meets $\mathcal{P}$ in at most $n - d$ points 
and let $G$ be a graph on $\mathcal{P}$ with $\iota(G) \geq n - d + 1$. 
Then the set \[\bigcup_{ P_iP_j \in E} \langle P_i, P_j \rangle\] 
 is a strong blocking set.
\end{cor}

A set of $n$ points in $\mathrm{PG}(k - 1, q)$ that meets every hyperplane in at most $n - d$ points is equivalent to a (nondegenerate) $[n, k, d]_q$ code (see Lemma~\ref{lem:intersection}).
Therefore, if we pick our set $\mathcal{P}$ corresponding to an asymptotically good linear code and our graph $G$ to be a bounded degree graph with $\iota(G) \geq n - d + 1$, then we can get a construction where for any fixed $q$ the size of the strong blocking set is linear in $q k$. 

\begin{cor}
Say there exists a family of codes $C$ with lengths $\{n_i\}_{i \geq 1}$, rate $R$ and relative distance $\delta$, and a family of graphs $G_i$ on $n_i$ vertices with maximum degree $\Delta$ and $\iota(G_i) > (1 - \delta) n_i$.
Then there exists strong blocking sets of size $\Delta n_i/2$ in $\mathrm{PG}(Rn_i - 1, q)$, for al $i \geq 1$. 
\end{cor}

It is shown in~\cite{abdn2022} that constant degree expander graphs $G$ on $n$ vertices have the property that $\iota(G) = \Omega(n)$. 
The explicit constructions of these graphs~\cite{lubotzky1988ramanujan} and asymptotically good linear codes~\cite{tsfasman2013algebraic} thus imply Theorem~\ref{thm:abdn}, and we refer to~\cite{abdn2022} for the best values of the constant $c$ obtained from this construction.

\section{Linear trifferent codes}
\label{sec:trifferent}

Recall that a linear subspace $C$ of $\F_3^n$ is called a \textit{linear trifferent code} if for all distinct $x, y, z \in C$ there exists a coordinate $i$ such that $\{x_i, y_i, z_i\} = \F_3$. The maximum size of a linear trifferent code is denoted by $T_L(n)$.

We say that an (affine) $2$-blocking set $S\subseteq \F_3^k$ is \textit{symmetric} if it is of the form $S=\{\vec{0}\}\cup B\cup -B$ for some set $B\subseteq \F_3^k$. So by Lemma~\ref{lem:equivalence_strong_affine}, a set $S\subseteq \F_3^k$ is a symmetric $2$-blocking set if and only if it is of the form $S=\cup_{\ell\in \mathcal{L}}\ell$ for a strong blocking set $\mathcal{L}\subseteq \mathrm{PG}(k-1,3)$.

In~\cite{Pohoata2022} it was shown that a linear trifferent code of dimension $k$ in $\F_3^n$ gives rise to a symmetric $2$-blocking set in $\F_3^k$. We prove that this relation goes both ways.
\begin{thm}\label{thm:triff_to_block}
Let $G\in \F_3^{k\times n}$ be a matrix of rank $k$. 
Let $B$ be the set of columns of $G$ and let $C$ be the row-space of $G$. 
Then $C$ is a linear trifferent code in $\F_3^n$ if and only if $\{\vec{0}\}\cup B\cup -B$ is a $2$-blocking set in $\F_3^k$. 
\end{thm}
\begin{proof}
We first note that every $x\in C$ is of the form $x=u^\transp G$ for a unique $u\in \F_3^k$ (as $G$ has rank $k$) and the $i$-th coordinate of $x$ is equal to $u^\transp b$, where $b$ is the $i$-th column of $G$. 

For the forward implication, suppose that $\{\vec{0}\}\cup B\cup -B$ is a 2-blocking set. Since $C$ is a linear subspace, we only need to show the trifferent property for triples of distinct vectors $\vec{0},c,c'\in C$. Writing $c=u^\transp G$ and $c'=v^\transp G$, it suffices to show that there is a column $b$ of $G$ such that $\{u^\transp b, v^\transp v\}=\{-1,1\}$. Since $u$ and $v$ are distinct and nonzero, the set $S = \{x\in \F_3^m: u^\transp x=1, v^\transp x=-1\}$ contains an affine subspace of codimension~$2$. Since $\{\vec{0}\}\cup B\cup -B$ is a 2-blocking set it must intersect $S$. Hence, since $S$ does not contain the zero vector, it contains a vector $x\in B\cup -B$. If $x\in B$ we can take $b=x$ and otherwise we can take $b=-x$.

For the backward implication, suppose that $C$ is a trifferent code. Let $V\subseteq \F_3^k$ be an affine subspace of codimension~$2$ with $0\not \in V$. It suffices to show that there is a $b\in B$ with $b\in V$ or $-b\in V$. 
We can write $V=\{x\in \F_3^k: u^\transp x=1, v^\transp x=-1\}$ for certain linearly independent $u,v\in \F_3^k$. Applying the trifferent property to the three distinct vectors $\vec{0}$, $u^\transp G$, $v^\transp G$, there is a $b\in B$ such that $\{u^\transp b, v^\transp b\}=\{-1,1\}$. Hence, $b\in V$ or $-b\in V$ as required.     
\end{proof}

From Lemma~\ref{lem:equivalence_strong_affine}, Theorem~\ref{thm:equivalence_minimal_strong} and Theorem~\ref{thm:triff_to_block}, we can deduce that a linear code $C \subseteq \F_3^n$ is minimal if and only if it is trifferent. 
We give a direct short proof of this equivalence.

\begin{thm}\label{thm:equivalence_triff_min}
A linear code $C \subseteq \F_3^n$ is trifferent if and only if it is minimal. 
\end{thm}
\begin{proof}
First suppose that $C$ is not minimal. Then there exist distinct nonzero codewords $u,v\in C$ with $\supp(u)\subsetneq \supp(v)$. 
Let $w=-u$, which must also lie in $C$ as $C$ is linear. 
Then there is no index $i$ such that $\{u_i,v_i,w_i\}=\F_3$, since $v_i=0$ implies $u_i=w_i=0$ and $u_i=0\iff w_i=0$. We conclude that $C$ is not a trifferent code.

For the other direction, suppose that $C$ is not a trifferent code. Let $u,v,w\in C$ be distinct elements such that $\{u_i,v_i,w_i\}\neq \F_3$ for every index $i$. We may assume that $w=\vec{0}$ (otherwise replace $u,v,w$ by $u-w,v-w,w-w$ respectively). Note that this implies that $u,v\neq \vec{0}$. Since there is no index $i$ for which $\{u_i,v_i\}=\{1,2\}$, we have $\supp(u+v)=\supp(u)\cup \supp(v)$. Moreover, since $u$ and $v$ are distinct, $\supp(u)\neq \supp(v)$, which implies that $\supp(u)\subsetneq \supp(u+v)$ or $\supp(v)\subsetneq \supp(u+v)$. We conclude that $C$ is not a minimal code. 
\end{proof}

The minimum size of a subset $B\subseteq \F_q^k$ such that $\cup_{\zeta \in \mathbb{F}_q} \zeta B$ is a $2$-blocking set in $\F_q^k$ is equal to 
$b^*_q(k, 1)$ by Lemma~\ref{lem:equivalence_strong_affine}. Moreover, we have $b_q(k, 2) \leq (q - 1) b^*_q(k, 1) + 1$. 
Note that $b_3^*(k, 1)$ is the smallest size of a symmetric $2$-blocking set in~$\F_3^k$ as defined before. 

\begin{cor}
\label{cor:linear_trifferent}
For all positive integers $k,n$, we have
\[
T_L(n)\geq 3^k \iff b^*_3(k,1)\leq n.
\]
\end{cor}
\begin{proof}
For the forward implication, suppose that $C\subseteq \F_3^n$ is a linear trifferent code with $|C|\geq 3^k$. Let $G\in \F_3^{k\times n}$ be a matrix whose rows are $k$ linearly independent vectors from $C$. Let $B$ be the set of columns of $G$. Then $|B|=n$ and $\{\vec{0}\}\cup B\cup -B$ is a 2-blocking set in $\F_3^k$. So $b^*_3(k,1)\leq n$. 

For the backward implication, let $\{\vec{0}\}\cup B\cup -B$ be a $2$-blocking set in $\F_3^k$ with $|B|=b^*_3(k,1)\leq n$. 
Let $G\in \F_3^{k\times n}$ be the matrix with the elements of $B$ as columns (each element of $B$ occurs at least once as a column). Note that since $\{\vec{0}\}\cup B\cup -B$ is a $2$-blocking set, it is not contained in a linear hyperplane of $\F_3^k$. 
So the matrix $G$ has rank $k$. It follows that the row space of $G$ is a $k$-dimensional linear trifferent code, so $T_L(n)\geq 3^k$.
\end{proof}

\subsection{Lower bound}
By Theorem~\ref{thm:blocking_upper} we have $b_3(k,2)\leq 1+16k/\log_3(81/25)$. Moreover, the obtained random $2$-blocking set is a union of lines by construction. 
Therefore, we also have $b_3^*(k, 1) \leq 8k/\log_3(81/25)$. 
Let $k=\lfloor n\frac{\log_3(81/25)}{8}\rfloor$. 
Then $b^*_3(k,1)\leq n$. 
From Corollary~\ref{cor:linear_trifferent}, it follows that 
\[T_L(n)\geq 3^k\geq \tfrac{1}{3}3^{n\frac{\log_3(81/25)}{8}}=\tfrac{1}{3}(9/5)^{n/4},\] 
thus proving the lower bound in Theorem~\ref{thm:trifferent_bounds}.

\subsection{Upper bound}

Let $C \subseteq \mathbb{F}_3^n$ be a linear trifferent code of size $k = T_L(n)$. 
Then by Theorem~\ref{thm:triff_to_block}, $C$ gives rise to a set $B \subseteq \F_3^k$ of size $n$ such that $\{\vec{0}\} \cup B \cup -B$ is an affine $2$-blocking set. 
The equivalence given in Lemma~\ref{lem:equivalence_strong_affine} shows that we thus have a strong blocking set of size $m \leq n$ in $\mathrm{PG}(k - 1, 3)$. 
A computer calculation shows that $c_3 > 1.1375$, so by Theorem~\ref{thm:strong_lower} we have $n \geq m > 4.55(k-1)$ for sufficiently large $k$.
Therefore, $T_L(n) < n/4.55 + 1$ for sufficiently large $n$, thus proving the upper bound in Theorem~\ref{thm:trifferent_bounds}.

\subsection{Explicit construction}

The explicit construction outlined in Section~\ref{sec:explicit} gives us an affine $2$-blocking set in $\mathbb{F}_3^k$ of size $8ck + 1$.
This gives us an explicit construction of linear trifferent codes of length $n$ and dimension at least $\frac{n}{4c}$ because the construction is from a strong blocking set, and hence a union of lines through the origin. 
The best constant $c$ that we get from~\cite{abdn2022} is not good enough to improve the construction from~\cite{wx2001}, which has size dimension $n/112$. 
However, the construction in~\cite{bartoli2021small} does manage to improve the state of the art for $q = 3$. 
In particular, Corollary 3.3 in~\cite{bartoli2021small} (with $q_0 = 3$) implies that there is an explicit construction of strong blocking sets of size at most $48 k_i$ in $\mathrm{PG}(k - 1, q)$, for an infinite sequence of $\{k_i\}_{i \geq 1}$ (see~\cite[Theorem 3.2]{bartoli2021small} for the exact value of $k_i$). 
Therefore, we get an explicit construction of linear trifferent codes of length $n_i \leq 48 k_i$ and dimension $k_i$.
We now improve this explicit construction by computing $T_L(n)$ for some small values of $n$. 
For fixed dimension $k$, the minimum size of a symmetric $2$-blocking set in $\F_3^k$ can be found using integer linear programming:
\begin{equation*}
2b^*_3(k,1)+1=
\begin{array}[t]{ll@{}ll}
\text{min}  & \sum_{v\in \F_3^k} x_v&\\
\text{s.t.}& \sum_{v\in W} & x_v\geq 1,  &\forall W\subseteq \F_3^k\text{ co-dim $2$ affine subspace}\\
& \sum_{v\in H} & x_v\geq 2k-1,  &\forall H\subseteq \F_3^k\text{ affine hyperplane}\\
&&x_{\vec{0}}=1\\
&&x_v-x_{-v}=0&\forall v\in \F_3^k\setminus\{\vec{0}\}\\
                 &&x_{v} \in \{0,1\}, &\forall v\in \F_3^k
\end{array}
\end{equation*}
The inequalities $\sum_{v\in H} x_v\geq 2k-1$ are redundant and follow from the bound 
$b_q(k,1)\geq (q-1)(k-1)+1$. However, adding these inequalities seems to significantly speed up computations\footnote{Using Gurobi 10.0 the values for $k\leq 5$ were found within a few minutes on a personal computer. We were not able to compute the exact value of $b^*_3(6,1)$.}. We obtained the following explicit values for small $k$:
\begin{center}
\begin{tabular}{c|ccccc}
$k$&2&3&4&5&6\\\hline
$b^*_3(k,1)$&4&9&14&19&22--24
\end{tabular}
\end{center}
By Theorem~\ref{thm:triff_to_block}, this implies that 
\[
T_L(n)=\begin{cases}
3^1 & \text{for }n\leq 3,\\
3^2 & \text{for }4\leq n\leq 8,\\
3^3 & \text{for }9\leq n\leq 13,\\
3^4 & \text{for }14\leq n\leq 18,\\
3^5 & \text{for }19\leq n\leq 21.
\end{cases}
\]

In particular, we have found a linear trifferent code of dimension $6$ and length $24$, that we now use to find our general explicit construction. 

\begin{proof}[Proof of Theorem~\ref{thm:explicit_trifferent}]
    Let $C_\mathrm{in}$ be the $[24, 6]_3$ trifferent code defined by the following generator matrix.

\[
    \begin{bmatrix}
101111101000011011111111 \\
012200000101120112210201 \\
121212202202012220112100 \\
110110100011120021110011 \\
220002101012122202102210 \\
010002111020111202200000
\end{bmatrix}
\]
Let $C_\mathrm{out}$ be an explicit infinite family of $[N, R N, \delta N]_{3^6}$ codes with $\delta = 2/3$ and $R = 1/3 - 1/(\sqrt{3^6} - 1)$. Such a family can be constructed using algebraic-geometric codes \cite{tsfasman2013algebraic}. 
Then by the argument in \cite{alon1986explicit} we know that $C_\mathrm{out}$ has the trifference property. 
Therefore, the concatenation $C_\mathrm{out} \circ_\pi C_\mathrm{in}$ is an $\mathbb{F}_3$-linear trifferent code of length $n = 24N$ and dimension $k = 6RN$, and $k/n = R/4 = 23/312.$
\end{proof}

\section{Conclusion}
In this paper, we established new connections between affine blocking sets, strong blocking sets, and trifferent codes.
We obtained new bounds on affine blocking sets which improve the state of the art for bounds on the latter two objects as well. 
Moreover, using the recent explicit constructions of strong blocking sets, we gave new explicit constructions of trifferent codes, beating the current bound. 
Despite this progress, many interesting problems remain open. 

Recall that $b_q(k, s)$ denotes the smallest size of an affine $s$-blocking set in $\F_q^k$. 
While we can prove upper bounds on $b_q(k, s)$ that, for any fixed $s$, are only a constant factor away from the lower bound given in~\eqref{eq:blocking_lb_basic}, the problem of determining the asymptotics of $b_q(k, s)$ when $s$ varies with $k$ is wide open. 

\begin{ques}
What is the asymptotic growth of $b_q(k, s)$, for fixed $q$, $s = \Theta(k)$, and $k \rightarrow \infty$? 
\end{ques}

While we could improve the lower bounds on strong $(s - 1)$-blocking sets, we are unable to improve the lower bounds on affine $s$-blocking sets given in~\eqref{eq:blocking_lb_basic}, for fixed $s$ and $q, k$ large. 
In particular, we ask the following. 
\begin{ques}
For every fixed prime power $q$, is there a constant $C_q > 1$ such that $b_q(k, 2) \geq C_q q^2 k$, for large enough $k$?
\end{ques}

Finally, we proved a lower bound on linear trifferent codes that is asymptotically equal to the best lower bound on trifferent codes. 
Our lower bound is based on the new upper bound on strong blocking sets, obtained by picking a random set of planes through the origin in $\F_q^k$. 
Any improvement in our argument would be very interesting, as it might lead to a breakthrough for the trifference problem. 

\begin{ques}
Is $\liminf_{n \rightarrow \infty} \frac{\log_3(T_L(n))}{n} > \frac{\log_3(9/5)}{4}$?
\end{ques}

The data on $T_L(n)$ for small $n$, that we computed in Section~\ref{sec:explicit}, suggests that $\lim_{n \rightarrow \infty}\log_3(T_L(n))/n = 1/5$, but we are not too confident to make that conjecture. 

While it is natural to extend the notion of linear perfect $3$-hash codes to linear perfect $q$-hash codes, with $q > 3$, the following argument shows that these objects are trivial.
Let $\F_q$ be a finite field with $q > 3$. 
We show that there is no linear perfect $q$-hash code $C\subseteq \F_q^n$ of dimension $\geq 2$. 
Suppose that $C\subseteq \F_q^n$ is a subspace of dimension $\geq 2$, and let $u, v$ be two linearly independent vectors.  Write $\F_q\setminus\{0\} = \{\zeta^i : 0 \leq i \leq q - 2\}$ and consider the set of vectors $\{u, \zeta u\}\cup\{v, \zeta v, \zeta^2 v, \ldots, \zeta^{q-3} v\}$. Say there is a coordinate $i$ where they are all distinct. Since $q>3$, $u_i$ and $v_i$ must both be nonzero. But we then have $q$ distinct nonzero coordinates, which is impossible. So $C$ cannot be a perfect $q$-hash code.

Therefore, it is more sensible to study linear codes $C \subseteq \F_q^n$, with the property that for any $t$ distinct codewords in $C$, there is a coordinate where they are all pairwise distinct, for some parameter $t < q$.
Such codes have been studied in the literature under the name of linear perfect hash families, and in fact it can be shown that these codes cannot exist for $t \geq c \sqrt{q}$  for some constant $c$ (see \cite[Section 5]{blackburn1998optimal}). 
In \cite{wx2001}, an explicit construction is given that has dimension at least a linear function of $n$, when $t = O(q^{1/4})$.
It would be interesting to improve these results on perfect hash families in view of our work.

\section*{Acknowledgement}
The research of Jozefien D’haeseleer is supported by the FWO (Research Foundation Flanders).
We thank Alessandro Neri and the anonymous reviewer for their helpful comments. 

{\footnotesize
\bibliographystyle{plain}
\bibliography{sample.bib}}

\appendix
\section{Appendix}
\label{sec:appendix}

\begin{lemma}
\label{lem:c_q}
For every prime power $q$, $M_q((q - 1)/(q + 1)) < 1/(q + 1)$.
\end{lemma}

\begin{proof}
For $q = 2$, one can easily verify that $M_2(1/3) < 1/3$, so we assume $q \geq 3$ for the rest of the argument. 
Let $\delta := \frac{q-1}{q+1}$.
We have, after a little algebraic manipulation,
\begin{align*}
x_{\delta} &:= 1 - \frac{1}{q} - \left(1 - \frac{2}{q}\right)\delta - \frac{2}{q}\sqrt{(q-1)\delta(1 - \delta)} \\
&= \frac{q-1}{q(q+1)}\left(3 - 2\sqrt{2} \right).
\end{align*}
Let $w : = 3 - 2\sqrt{2}$, so that $x_{\delta} = w \frac{q-1}{q(q+1)}$, and so one may verify that
\begin{equation}\label{eqn:calc1}
-\frac{(q-1)}{q} w\log_q w  < 0.19 
\end{equation}
for $q \geq 3$. 
We also have the following inequality for every $q \geq 3$,
\begin{equation}\label{eqn:calc3}
\frac{q-1}{q}\log_q(q(q+1)) = \frac{2(q-1)}{q} + \frac{(q-1)}{q}\log_q(1 + 1/q) <\frac{2(q-1)}{q} + \frac{q-1}{q^2} < 2.
\end{equation}
Here we have used the fact that 
\[
\log_q(1 + 1/q)=\frac{\ln(1 + 1/q)}{\ln q} < \frac{1}{q},
\]
for $q \geq 3$. 
Finally, note that 
\begin{equation}\label{eqn:calc2}
x_{\delta} = \frac{w(q-1)}{q(q+1)} < \frac{w}{q+1}
\end{equation}

This gives us
\begin{align*}
M_q\left(\frac{q-1}{q+1} \right) & = x_{\delta} \log_q(q-1) - x_{\delta}\log_qx_{\delta} - (1 - x_{\delta})\log_{q}(1 - x_{\delta})\\
&\leq x_{\delta} \log_q(q-1) - x_{\delta}\log_qx_{\delta} - \log_{q}(1 - x_{\delta})\\
& \leq x_{\delta}\log_q\left(\frac{q-1}{x_{\delta}}\right) + 2x_{\delta}\\
& = -\frac{w(q-1)}{q(q+1)}\log_q{w} + \frac{w(q-1)}{q(q+1)}\log_q(q(q+1)) +  2x_{\delta} \\
& \leq \frac{0.19 + 4w}{q+1} \\
& < \frac{1}{q + 1}. 
\end{align*}
For the second inequality, we used the fact that $-\log_q(1 - x) < 2x$ for $x < 0.5$,  and for the second last inequality, we used~\eqref{eqn:calc1},\eqref{eqn:calc3}, and~\eqref{eqn:calc2}. 

\end{proof}

\begin{cor}
 There is a unique solution $c_q$ to $M_q((q - 1)/(x(q + 1)) = 1/(x(q+1))$, and $x \geq 1$. 
 Moreover, $c_q > 1$. 
\end{cor}
\begin{proof}
The function $f(y) \coloneqq M_q((q - 1)y/(q + 1))$ is a continuous strictly decreasing function for $0 \leq y \leq 1$, with $f(0) = 1$, and the function $g(y) \coloneqq y/(q + 1)$ is a continuous strictly increasing function, with $g(0) = 0$. 
We have just shown that $g(1) > f(1)$, and thus there must exist a unique $0 < y_q < 1$ for which $f(y_q) = g(y_q)$. 
Therefore, $c_q = 1/y_q > 1$. 
\end{proof}

\begin{remark}
The proof also shows that $c_q$ is the maximum $x \geq 1$ for which $M_q((q - 1)/(x(q+1)) \leq 1/(x(q+1))$. 
Using similar arguments, one may show that $c_q > 1 + 1/(2000q)$, for all prime powers $q$. 
\end{remark}

\end{document}